\documentclass[a4paper,12pt]{amsart}
\usepackage{amsmath}
\usepackage{amssymb}
\usepackage{mathrsfs}
\makeatletter
\newcommand*{\rom}[1]{\expandafter\@slowromancap\romannumeral #1@}
\makeatother
\usepackage{enumerate}
\usepackage{ifthen}
\usepackage{graphicx}
\usepackage[T1]{fontenc} 
\usepackage{tabularx}
\usepackage{multirow}
\usepackage{color,soul}
\usepackage{tikz}
\usepackage{pgfplots}
\usepackage{upgreek}

\nonstopmode \numberwithin{equation}{section}
\newtheorem{thm}{Theorem}[section]
\newtheorem{cor}{Corollary}[section]
\newtheorem{lem}{Lemma}[section]

\theoremstyle{definition}

\newtheorem{rem}{Remark}[section]

\newtheorem{innercustomthm}{Theorem}
\newenvironment{customthm}[1]
  {\renewcommand\theinnercustomthm{#1}\innercustomthm}
  {\endinnercustomthm}

\newcounter{minutes}
\divide\time by 60
\newcounter{hours}
\multiply\time by 60
\addtocounter{minutes}{-\time}

\newcounter {own}
\def\theown {\thesection       .\arabic{own}}

{\qed\bigskip}

\newcounter{alphabet}
\newcounter{tmp}

\begin{document}
\title{Bloch and Landau constants for meromorphic functions}

\author{Md Firoz Ali}
\address{Md Firoz Ali,
 National Institute Of Technology Durgapur, West Bengal, India}
\email{ali.firoz89@gmail.com, fali.maths@nitdgp.ac.in}

\author{Shaesta Azim }
\address{Shaesta Azim, National Institute Of Technology Durgapur, West Bengal, India}
\email{taefaazim547@gmail.com}


\def\thefootnote{}

\begin{abstract}

Let $\mathcal{M}_1(\lambda)$ be the class of all meromorphic functions $f$ in the unit disk $\mathbb{D}=\{z\in\mathbb{C}\}: |z|<1$ having a simple pole at $\lambda \in \overline{\mathbb{D}} \setminus \{0\}$ and satisfying the normalization $f'(0)=1$. Let $B(\lambda)$ and $L(\lambda)$ denote the Bloch and Landau constants, respectively, for this class. In this article, we first show that the Bloch constant $B(1)$ and the Landau constant $L(1)$ are infinite. Using these results and a conformal mapping technique, we establish that $B(p)$ and $L(p)$ are likewise infinite for any $p \in (0,1)$, thereby refuting a recent conjecture. Finally, we extend our study to the class of meromorphic functions having two simple poles and prove that their associated Bloch and Landau constants also remain infinite.

\end{abstract}

\thanks{}

\maketitle
\pagestyle{myheadings}
\markboth{Md Firoz Ali, and Shaesta Azim}{Bloch and Landau constants for meromorphic functions}

\section{Introduction}
Let $\mathbb{C}$ denote the complex plane. For $a \in \mathbb{C}$ and $r > 0$, we denote the disk centered at $a$ with radius $r$ by $\mathbb{D}(a, r) = \{z \in \mathbb{C} : |z - a| < r\}$, and its boundary by $\partial \mathbb{D}(a, r) = \{z \in \mathbb{C} : |z - a| = r\}$. We let $\mathbb{D} = \mathbb{D}(0, 1)$ denote the unit disk. Let $\mathcal{A}(\mathbb{D})$ be the class of all analytic functions on $\mathbb{D}$, and let $\mathcal{A}_1 = \{f \in \mathcal{A}(\mathbb{D}) : f'(0) = 1\}$.\\

Finding the exact value of the Bloch constant remains one of the longstanding open problems in geometric function theory. This problem originated with Bloch's discovery that for any function $f \in \mathcal{A}_1$, there exists a subdomain of $\mathbb{D}$ on which $f$ is univalent and whose image under $f$ is a disk, known as a schlicht disk. In 1929, Landau \cite{Landau-1929} defined the Bloch constant $B$ as $B = \inf\{B_f : f \in \mathcal{A}_1\}$, where $B_f$ denotes the radius of the largest schlicht disk contained in $f(\mathbb{D})$. In the same article, Landau also introduced two related constants: the Landau constant, $L = \inf\{L_f : f \in \mathcal{A}_1\}$, where $L_f$ denotes the radius of the largest disk contained in $f(\mathbb{D})$; and the univalent Bloch constant,$$\mathcal{U} = \inf \{ B_f : f \in \mathcal{A}_1 \text{ and } f \text{ is univalent} \}.$$ Landau \cite{Landau-1929} provided numerical estimates for these constants and proved that $0.39 \leq B \leq L < 0.555$ and $\mathcal{U} > 0.566$.\\

A function $f \in \mathcal{A}(\mathbb{D})$ is classified as a Bloch function if its semi-norm, defined by
$$\|f\| := \sup_{z \in \mathbb{D}} (1 - |z|^2) |f'(z)|,$$
is finite. In 1942, Seidel and Walsh \cite{Seidel-Walsh-1942} proved that $B_f$ is finite if  and only if $||f||$ is finite, which was first pointed out by Pommerenke  \cite{Pommerenke-1970}. In the same paper, Pommerenke  \cite{Pommerenke-1970} introduced the locally univalent Bloch constant $B_{\ell}$, defined as
$$B_{\ell}=\inf\{{B_{f}}:f\in \mathcal{A}_{1},~\text{and}~f'(z)\neq 0 ~\text{for}~z\in \mathbb{D}\}$$
and demonstrated that $B_{\ell}>1/2$.
These constants satisfy the fundamental inequality chain $B\leq B_{\ell}\leq L \leq \mathcal{U}$. For more details on Bloch function we refer to \cite{Cima-1979}.\\

Historically, the quest for the exact values of these constants has generated significant literature. In 1937, Ahlfors and Grunsky \cite{Ahlfors-Grunsky-1937} derived an upper bound for the Bloch constant given by
$$B\leq\dfrac{1}{\sqrt{1+\sqrt{3}}}\dfrac{\Gamma(1/3)\Gamma(11/12)}{\Gamma(1/4)}\approx 0.4719 ,$$
a value they conjectured to be the exact value of $B$. Subsequently, Ahlfors \cite{Ahlfors-1938} refined the lower bounds for both $B$ and $L$, proving that $B \geq \sqrt{3}/4$ and $L \geq 1/2$. In 1943, Rademacher \cite{Rademacher-1943} further improved the upper bound for $L$:
$$L\leq\dfrac{\Gamma(1/3)\Gamma(5/6)}{\Gamma(1/6)}\approx 0.5433,$$
and likewise conjectured its sharpness. Chen and Gauthier \cite{Chen-Gauthier-1996} established the current best lower bound for $B$ as
$$B>\dfrac{\sqrt{3}}{4}+2\times10^{-4}.$$
More recently, in 2004, Chen and Shiba \cite{Chen-Shiba-2004} improved the lower bound for $L$ and proved $L>\dfrac{1}{2}+2\times10^{-8}$. In 2009, Carroll and Cerda \cite{Carrol-Cerda-2009} obtained the upper bound $\mathcal{U}\leq 0.6563937$ and in the same year, Skinner \cite{Skinner-2009} obtained the corresponding lower bound  $\mathcal{U}>0.5708858$ of $\mathcal{U}$.
These estimates are best available in the literature.
The best known lower bound for locally univalent function is $B_{\ell}>\dfrac{1}{2}+2\times10^{-8}$, which was given by Chen and Shiba \cite{Chen-Shiba-2004} in 2004.\\

The study of Bloch constants within the context of meromorphic functions represents another significant frontier in geometric function theory. In 1982, Minda \cite{Minda-1982} established that the Bloch constant for the family of locally univalent meromorphic functions on the complex plane $\mathbb{C}$ is exactly $\pi/2$. In the same study, Minda demonstrated that the Bloch constant for the broader class of all meromorphic functions on $\mathbb{C}$ lies within the interval $[\pi/3, 2\arctan(1/\sqrt{2})]$. This problem was later resolved by Bonk and Eremenko \cite{Bonk-Eremenko-2000} in 2000, who proved that the precise value of the Bloch constant for this class is $\arctan(\sqrt{8})$.\\

Let $\mathcal{M}(\lambda)$ be the class of all meromorphic functions $f$ in the unit disk $\mathbb{D}$ having a simple pole at $\lambda\in\overline{\mathbb{D}}\setminus\{0\}$. We define $\mathcal{M}_1(\lambda)$ as the subclass of $\mathcal{M}(\lambda)$ subject to the normalization $f'(0)=1$. For any $f \in \mathcal{M}_1(\lambda)$, let $L_f(\lambda)$ and $B_f(\lambda)$ denote the radii of the largest disk and the largest schlicht disk, respectively, contained within the image $f(\mathbb{D})$. The Landau and Bloch constants for the class $\mathcal{M}_1(\lambda)$ are then defined by
$$L(\lambda) = \inf \{ L_f(\lambda) : f \in \mathcal{M}_1(\lambda) \} \quad \text{and} \quad B(\lambda) = \inf \{ B_f(\lambda) : f \in \mathcal{M}_1(\lambda) \}.$$
We observe that if $\lambda = \rho e^{i\alpha}$ and $f \in \mathcal{M}_1(\lambda)$, then the rotation $f_\alpha(z) := e^{-i\alpha} f(e^{i\alpha}z)$ belongs to the class $\mathcal{M}_1(\rho)$. Since $L_f(\lambda) = L_{f_\alpha}(\rho)$ and $B_f(\lambda) = B_{f_\alpha}(\rho)$, it follows that $L(\lambda) = L(\rho)$ and $B(\lambda) = B(\rho)$. Consequently, without loss of generality, we can assume that $0< \lambda\le 1$.\\

For $\lambda = p \in (0,1)$, we adopt the notation $\mathcal{A}(p) := \mathcal{M}(p)$ and $\mathcal{A}_1(p) := \mathcal{M}_1(p)$. The notion of the Landau and Bloch constants for functions in $\mathcal{A}_1(p)$ were introduced by Bhowmik and Sen \cite{Bhomik-Sen-2023}.
Indeed, Bhowmik and Sen \cite{Bhomik-Sen-2023} defined ${L}(p)$ and $B(p)$ for the class $\mathcal{A}(p)$ with $f'(0)\ne 0$. But it is analytically more convenient to define them for the normalized subclass $\mathcal{A}_1(p)$ to maintain consistency with the classical class $\mathcal{A}_1$.
For the class $\mathcal{A}(p)$, Bhowmik and Sen \cite{Bhomik-Sen-2023a} proved that
$$B(p)\geq(8-\sqrt{63})^2 p^2|f'(0)|\quad\text{and}\quad L(p)\geq\dfrac{(9-4\sqrt{5})p^2|f'(0)|}{8}.$$
Subsequently, these bounds were refined in \cite{Bhomik-Sen-2023}, yielding the improved estimates detailed below.

\begin{customthm}{A}\label{s1-thm-A}\cite{Bhomik-Sen-2023}
Let $B$ and $L$ be the Bloch and Landau constant for the class $\mathcal{A}_{1}$. Then  $$B(p)\geq\frac{4p|f'(0)|B}{(1+p)^2}\quad\text{and}\quad L(p)\geq\frac{4p|f'(0)|L}{(1+p)^2}.$$
\end{customthm}

In Theorem \ref{s1-thm-A}, if one allow $p\to 1^-$ (with the normalization $f'(0)$=1) then $B(1)\geq B$ and $L(1)\geq L$. Moreover, Bhowmik and Sen \cite{Bhomik-Sen-2023} remarked that one can easily show that $B(1)= B$ and $L(1)= L$ which led them to conjecture that the bounds provided in Theorem \ref{s1-thm-A} constitute the exact values for $B(p)$ and $L(p)$.\\

The study of Bloch and Landau constants for harmonic and logharmonic functions has gained significant popularity in recent years. For a detailed review of the literature, we refer to the following papers and the references therein \cite{Bhomik-Sen-2025, Chang-Ponnusamy-Qiao-Qie-2026, Liu-Luo-Ponnusamy-2024, Sen-2026}.\\

In this article, we first establish that the Bloch constant $B(1)$ and the Landau constant $L(1)$ are infinite, and using this, we show that the conjecture by Bhowmik and Sen \cite{Bhomik-Sen-2023} concerning the exact values of the Bloch constant $B(p)$ and the Landau constant $L(p)$ is not true. Furthermore, we extend our investigation of Bloch and Landau constants to the class of meromorphic functions which are analytic in the unit disk except for two prescribed simple poles.

\section{Bloch and Landau constants for the class $\mathcal{M}_1(\lambda)$}

To provide context for our main findings, we begin by examining a specific case that served as the primary motivation for this study. This case illustrates the geometric phenomena that our results generalize.\\

A function $f\in \mathcal{A}(\mathbb{D})$ having a simple pole on $\partial\mathbb{D}$ is said to be a concave function if the complement of $f(\mathbb{D})$ is an unbounded convex domain.
For $\alpha\in(1,2]$, let $Co(\alpha)$ denote the class of concave univalent functions with opening angle $\pi\alpha$ at infinity. A function $f\in\mathcal{A}(\mathbb{D})$ is in $Co(\alpha)$ if it satisfies the following conditions:
\begin{enumerate}[(i)]
    \item $f\in \mathcal{A}_1$ with $f(0)=0$ and $f(1)=\infty$,
    \item $f$ is univalent in $\mathbb{D}$ and $\mathbb{C}\setminus f(\mathbb{D})$ is a convex domain,
    \item The opening angle of at $\infty$ is less than or equal to $\pi\alpha,~\alpha\in (1,2]$.
\end{enumerate}
For this class, Avkhadiev and Wirths \cite{Avkhadiev-Wirths-2005} established that the intersection of all image domains is a specific half-plane:
$$\bigcap_{f\in Co(\alpha)}f(\mathbb{D})=\Omega_\alpha:=\left\{w:~\text{Re}~w>-\dfrac{1}{2\alpha}\right\}.$$
It follows directly from this result that every $f \in Co(\alpha)$ contains the half-plane $\Omega_\alpha$. Since $\Omega_\alpha$ contains disks of arbitrarily large radii, for instance, $\mathbb{D}(r, r) \subset \Omega_\alpha$ for any $r > 0$, the Bloch and Landau constants for the class $Co(\alpha)$ are necessarily infinite.\\

Our first result demonstrates that if a function $f$ is analytic in the unit disk $\mathbb{D}$ and has a simple pole at the boundary point $z=1$, then the radius of the largest schlicht disk contained in $f(\mathbb{D})$ is infinite.

\begin{thm}\label{s1-thm-002}
If $f \in \mathcal{M}_1(1)$ then the radii $B_{f}(1)$ and ${L}_{f}(1)$ are infinite.


\end{thm}

\begin{proof}
To show that $B_f(1)$ is infinite, it is sufficient to show that the Bloch semi-norm
$$||f||=\sup_{z\in\mathbb{D}}(1-|z|^2)|f'(z)|$$
is unbounded. If $f \in \mathcal{M}_1(1)$, $f$ has a simple pole at $z=1$ and consequently, in a deleted neighborhood of the pole, $f$ may be represented as
$$
f(z) = \frac{h(z)}{z-1}
$$
where $h$ is analytic at $z=1$ and $h(1)\neq 0$.
By evaluating the semi-norm along the positive real axis $t \to 1^-$, we obtain the following lower bound:
\begin{align*}
    ||f||
    &=\sup_{z\in\mathbb{D}}(1-|z|^2)|f'(z)|\\
    &\geq \limsup_{t \to 1^-} (1-t^2)|f'(t)|\\
    &=\limsup_{t \to 1^-}(1-t^2)\left|\dfrac{h'(t)(t-1)-h(t)}{(t-1)^2}\right|\\
    &=\limsup_{t \to 1^-}\dfrac{1+t}{1-t}\left|h'(t)(t-1)-h(t)\right|\\
    &= \infty.
\end{align*}
This implies that $B_f(1)$ is infinite. The infinitude of $L_f(1)$ follows immediately from the inequality $L_f(\lambda) \geq B_f(\lambda)$.

\end{proof}

\begin{cor}\label{s1-thm-005}
Then Bloch constant $B(1)$ and the Landau constant $L(1)$ for the class $\mathcal{M}_{1}(1)$ are infinite.
\end{cor}


\begin{cor}\label{s1-thm-012}
If $f \in \mathcal{A}_1$ is analytic in the unit disk $\mathbb{D}$ and has a simple pole at $\lambda \in \partial \mathbb{D}$, then the radius of the largest schlicht disk $B_f(\lambda)$ and the radius of the largest disk $L_f(\lambda)$ contained in the image $f(\mathbb{D})$ are infinite.
\end{cor}


Before proceeding to our subsequent findings, we examine another specific instance that provided the necessary intuition for the next theorem. Let $\mathcal{S}(p)$, $0<p<1$, be the class of all univalent functions in $\mathcal{A}(p)$ with the normalization $f(0)=0=f'(0)-1$.
If $f\in \mathcal{S}(p)$ then it is well known that \cite[Theorem 41, Page-249, Vol-II]{Goodman-1983} $f(\mathbb{D})$ contains the disjoint domains:
$$D_1=\left\{w\in\mathbb{C}:|w|<\dfrac{p}{(1+p)^2}\right\}~~\text{and}~~ D_2=\left\{w\in\mathbb{C}:|w|>\dfrac{p}{(1-p)^2}\right\}.$$
The existence of the exterior domain $D_2$ implies that $f(\mathbb{D})$ contains a neighborhood of infinity. Consequently, for any $R > 0$, one may select a center $a \in f(\mathbb{D})$ such that the distance $d(a, \partial D_2) > R$, ensuring that the disk $\mathbb{D}(a, R)$ is contained within $f(\mathbb{D})$. Since $f$ is univalent, every such disk is necessarily schlicht. It follows that for each $f \in \mathcal{S}(p)$, the radius $B_f(p)$ is infinite. By extension, the Bloch constant $B(p)$ and the Landau constant $L(p)$ for the class $\mathcal{S}(p)$ are not finite. Our next result demonstrates that this conclusion persists even when the requirement of univalency is omitted.

\begin{thm}\label{s1-thm-001}
For any $p \in (0,1)$, the Bloch constant $B(p)$ and the Landau constant $L(p)$ for the class $\mathcal{A}_{1}(p)$ are infinite.
\end{thm}

\begin{proof}
For $p \in (0,1)$, let $\Omega_p$ denote the unit disk slit along the real axis from $p$ to $1$, i.e.,
$$
\Omega_{p}=\mathbb{D}\setminus{[p,1)}.
$$
Let $k(z) = z/(1+z)^2$ denote the Koebe function (rotation) and $K$ its inverse. For $r = 4p/(1+p)^2$, the function $\eta(z) = K(rk(z))$ maps $\mathbb{D}$ conformally onto $\Omega_p$. Letting $\zeta(z)$ denote the inverse of $\eta(z)$, we observe that
$$
\zeta'(0) = \frac{(1+p)^2}{4p}.
$$

Let $f \in \mathcal{A}_1(p)$ and consider the restriction $f_1 \equiv f|_{\Omega_p}$. For each such $f_{1}$, there exist a function $g\in \mathcal{A}_{1}$ with the property that $g$ has a simple pole at the point $\zeta(p)\in \partial \mathbb{D}$ such that
\begin{equation} \label{s1-001}
f_{1}(z)=\dfrac{4p}{(1+p)^2}g\circ \zeta(z),~z\in\Omega_{p}.
\end{equation}
It follows from \eqref{s1-001} that the image of the slit domain $\Omega_{p}$ under $f_1$ satisfies
\begin{equation} \label{s1-003}
    f(\Omega_{p})=f_{1}(\Omega_{p})=\dfrac{4p}{(1+p)^2}g(\mathbb{D})=:G.
\end{equation}
Let $B_f(p)$ and $B_g$ denote the radii of the largest schlicht disks contained within $f(\mathbb{D})$ and $g(\mathbb{D})$, respectively. Since $g\in \mathcal{A}_{1}$ with $g$ having a simple pole at the point $\zeta(p)\in \partial \mathbb{D}$, it follows from Corollary \ref{s1-thm-012} that $B_{g}$ is infinite. Consequently, from \eqref{s1-003} we can conclude that $B_f(p)$ is also infinite. As this holds for all $f \in \mathcal{A}_1(p)$, the Bloch constant $B(p)$ is also infinite. The Landau constant $L(p)$ is not finite follows immediately.

\end{proof}

\begin{rem}
In Theorem \ref{s1-thm-001}, since $f$ is analytic in $\mathbb{D} \setminus \{p\}$ and the segment $(p, 1)$ constitutes a portion of the boundary $\partial \Omega_p$, it follows that the image $f((p, 1))$ is contained in either $G$ or its boundary $\partial G$. Consequently, the image of the unit disk under $f$ satisfies the following inclusion relation
\begin{equation}\label{s1-005}
G \subset f(\mathbb{D}) \subset \overline{G}.
\end{equation}

It is noteworthy that the univalency of the composition $g \circ \zeta$ is equivalent to the univalency of $g$. Let $B_{f_1}$ denote the radius of the largest schlicht disk contained within $f_1(\Omega_p)$. Therefore, from \eqref{s1-001} and \eqref{s1-003}, we can conclude that
$$B_{f_1}=\dfrac{4p}{(1+p)^2}B_{g}.$$
From \eqref{s1-005}, one can further conclude that
$$B_{f}(p)=B_{f_1}=\dfrac{4p}{(1+p)^2}B_{g}.$$

By a similar geometric argument, if $L_f(p)$ and $L_g$ denote the radii of the largest disks (not necessarily schlicht) contained within $f(\mathbb{D})$ and $g(\mathbb{D})$, respectively then applying \eqref{s1-003} and \eqref{s1-005}, one can conclude that
$$L_{f}(p)=\dfrac{4p}{(1+p)^2}L_{g}.$$
\end{rem}


\subsection{Bloch constant for the class $\Sigma$}

Let $\Sigma$ denote the class of all univalent functions defined on the punctured unit disk $\mathbb{D}^* = \mathbb{D} \setminus \{0\}$ having a simple pole at the origin with residue $1$. Each $g \in \Sigma$ admits a Laurent series expansion of the form
$$
g(z)=\frac{1}{z}+\sum_{n=0}^{\infty} b^nz^n, \quad 0<|z|<1.
$$
It is a fundamental result in geometric function theory that any $g \in \Sigma$ maps $\mathbb{D}^*$ onto the complement of a compact connected set $E$. Let $g \in \Sigma$ and let $g(\mathbb{D}^*) = \mathbb{C} \setminus E$. Since the set $E$ is compact, its complement is a neighborhood of infinity. Consequently, for any $R > 0$, one may select a center $a \in \mathbb{C} \setminus E$ such that $\text{dist}(a, E) > R$, ensuring that the disk $\mathbb{D}(a, R)$ is contained within the image $g(\mathbb{D}^*)$. Furthermore, the univalency of $g$ guarantees that every such disk is schlicht. It follows that for each $g \in \Sigma$, the radius of the largest schlicht disk $B_g$ is infinite. Accordingly, the Bloch constant $B$ and the Landau constant $L$ for the class $\Sigma$ are not finite.

\section{Bloch constant for functions with two simple poles}
Let $\mathcal{M}(\lambda, \mu)$ be the class of all meromorphic functions $f$ in the unit disk $\mathbb{D}$ having two simple poles at $z=\lambda,\mu\in\overline{\mathbb{D}}\setminus\{0\}$ and $\mathcal{M}_1(\lambda,\mu)$ be the subclass of all meromorphic functions $f$ in $\mathcal{M}(\lambda,\mu)$ with the normalization $f'(0)=1$. For $f\in \mathcal{M}_1(\lambda,\mu)$, let ${L}_{f}(\lambda,\mu)$ and $B_{f}(\lambda,\mu)$ denote the radius of the largest disk and the radius of the largest schlicht disk, respectively, lying in  $f(\mathbb{D})$. Further, let
\begin{align*}
L(\lambda,\mu) &=\inf\{{L_{f}(\lambda,\mu)}:f\in\mathcal{M}_1(\lambda,\mu)\}\\[2mm]
\text{and}~B(\lambda,\mu)&=\inf\{{B_{f}(\lambda,\mu)}:f\in\mathcal{M}_1(\lambda,\mu)\}
\end{align*}
be the Landau and Bloch constant, respectively for the class $\mathcal{M}_1(\lambda,\mu)$.\\


The following lemma establishes the conformal mapping properties of a class of domains with multiple slits, extending the classical results of Robinson \cite{Robinson-1935}.

\begin{lem}\label{s1-thm-010}
Let $k(z)=z/(1-z)^2$ denote the Koebe function and $K(z)$ its inverse. For $x \in (0,1)$ and $\theta \in [0, 2\pi)$, define the parameters
\begin{align}\label{s1-020}
\rho=\rho(x):=\dfrac{4x}{(1+x)^2},\quad \xi(\rho,\theta):=1-2\rho\sin^2\frac{\theta}{2}-2i\rho\sin^2\frac{\theta}{2}\sqrt{\dfrac{1}{\rho\sin^2\frac{\theta}{2}}-1}.
\end{align}
\begin{enumerate}[(i)]
\item The function $\omega_\rho(z) = -e^{i\theta} K(\rho k(z))$ maps the unit disk $\mathbb{D}$ conformally onto the slit disk $\mathbb{D} \setminus L(x, \theta)$, where $L(x, \theta) = \{ te^{i\theta} : x \leq t < 1 \}$. Furthermore, the boundary correspondences are given by
$$ \omega_\rho(-1)=x e^{i\theta}\quad\text{and}\quad \omega_\rho\left(\xi(\rho,\theta)\right)=-1.$$

\item For $p,p_1\in(0,1)$, let $r=\rho(p)$, $r_1=\rho(p_1)$ and $-K(rk(p_1e^{i\theta}))=qe^{i\phi}$. Then the composition function
$$h(z)=-K(rk(-e^{i\theta}K(r_1k(z))))=-K(rk(\omega_{r_1}(z)))$$
maps the unit disk $\mathbb{D}$ conformally onto the domain $\Omega(p, p_1, \theta) := \mathbb{D} \setminus \{ [p, 1) \cup C(p_1, \theta) \}$, where the second slit is defined by the curve $C(p_1, \theta) := \{ -K(r k(t e^{i\theta})) : p_1 \leq t < 1 \}$. Furthermore, the boundary correspondences are given by
$$ h(-1)=qe^{i\phi}\quad\text{and}\quad h\left(\xi(r_1,\theta)\right)=p.$$
\end{enumerate}
\end{lem}

\begin{proof}

$(i)~$ It follows from \cite[Lemma 1]{Robinson-1935}  that the function $K(\rho k(z))$ maps the unit disk onto $\mathbb{D}\setminus{(-1,-x)}$. For fix $\theta\in[0,2\pi)$, the function $\omega_\rho (z)=-e^{i\theta}K(\rho k(z))$ maps the unit disk $\mathbb{D}$ onto the domain $\mathbb{D}\setminus L(x, \theta)$, where
$L(x, \theta)=\big\{te^{i\theta}:{x}\leq t<1\big\}.$\\

The function $K\left(\rho k(z)\right)$ maps the point $-1$ to $-x$ which gives $\omega_\rho(-1)=x e^{i\theta}$. Furthermore, solving the boundary equation  $-e^{i\theta}K\left(\rho k(z)\right)=-1$ for $z$ yields the preimage $\omega_\rho\left(\xi(\rho,\theta)\right)=-1$, where $\xi(\rho,\theta)$ is defined in \eqref{s1-020}.\\

$(ii)~$ From the properties established in $(i)$, the auxiliary function $-e^{i\theta}K(r_1k(z))$ maps the unit disk $\mathbb{D}$ onto the domain $\mathbb{D}\setminus L(p_1, \theta)$.
Since the transformation $w \mapsto -K(rk(w))$ maps the unit disk $\mathbb{D}$ onto the disk slit along the real axis $\mathbb{D} \setminus [p, 1)$ and carries the radial segment $L(p_1, \theta)$ onto the curve $C(p_1, \theta)$, it follows that this secondary mapping transforms $\mathbb{D} \setminus L(p_1, \theta)$ onto the domain $\Omega(p, p_1, \theta) = \mathbb{D} \setminus \{ [p, 1) \cup C(p_1, \theta) \}$
Consequently, the composition function
$$h(z)=-K(rk(-e^{i\theta}K(r_1k(z))))=-K(rk(\omega_{r_1}(z)))$$
is a conformal mapping from $\mathbb{D}$ to $\Omega(p,p_1,\theta)=\mathbb{D}\setminus\big\{ [p,1)\cup C(p_1,\theta)\big\}.$\\

The boundary correspondences follow by direct evaluation. Since $\omega_{r_1}(-1) = p_1 e^{i\theta}$ and $-K(rk(p_1 e^{i\theta})) = q e^{i\phi}$, we obtain $h(-1) = q e^{i\phi}$. Similarly, using $\omega_{r_1}(\xi(r_1, \theta)) = -1$ and the fact that $-K(rk(-1)) = p$, we conclude that $h(\xi(r_1, \theta)) = p$.

\end{proof}
\begin{thm}\label{s1-thm-015}
For any distinct $\lambda,\mu\in\overline{\mathbb{D}}\setminus\{0\}$, the Bloch constant $B(\lambda,\mu)$ for the class $\mathcal{M}_1{(\lambda,\mu)}$ is infinite.
\end{thm}

\begin{proof}
If either pole $\lambda$ or $\mu$ is located on the boundary $\partial\mathbb{D}$, then by Corollary \ref{s1-thm-012}, the Bloch constant $B(\lambda,\mu)$ is infinite. We now assume that both poles reside within the punctured disk, i.e., $\lambda, \mu \in \mathbb{D} \setminus \{0\}$. The relative positions of these poles lead to two distinct cases.

\textbf{Case 1: The poles $\lambda$ and $\mu$ do not lie on the same radial line.} In this case, without the loss of generality, we assume $\lambda=p\in(0,1)$ and $\mu=q e^{i\phi}$ with $\phi\neq 0$. Let $r$ and the mapping $-K(rk(z))$ be defined as in Lemma \ref{s1-thm-010}, which transforms $\mathbb{D}$ onto $\mathbb{D} \setminus [p,1)$. We define $p_1 e^{i\theta}$ as the preimage of $\mu = q e^{i\phi}$ under this map. Consequently, the function $h(z)$ from Lemma \ref{s1-thm-010} maps the unit disk conformally onto the simply connected domain $\Omega(p,p_1,\theta) = \mathbb{D} \setminus \{ [p,1) \cup C(p_1,\theta) \}$.\\

If $f\in \mathcal{M}_1{(p,q e^{i\phi})}$ then the restriction $f_1\equiv f|_{\Omega(p,p_1,\theta)}$ is analytic. Let $H$ denote the inverse of $h$. Clearly $H(z)$ is a conformal mapping with
$$H(p)=\xi(r_1,\theta) \in \partial\mathbb{D}~\text{and}~H(q e^{i\phi})=-1,
$$
where $r_1$ and $\xi(r_1,\theta)$ are defined as in Lemma \ref{s1-thm-010}. The composition $f_1 \circ h$ is analytic in $\mathbb{D}$ except for the points where $h$ maps to the poles of $f_1$. By choosing a suitable scaling constant $c > 0$, we define $g = c(f_1 \circ h)$ such that $g'(0) = 1$. The function $g$ is analytic in $\mathbb{D}$ and has simple poles at the boundary points $H(p)\in \partial\mathbb{D}$ and $-1\in \partial\mathbb{D}$. From the representation
 \begin{equation}\label{3}
     f_1(z)=\dfrac{1}{c}g\circ H(z),\quad z\in \Omega(p,p_1,\theta).
 \end{equation}
It follows that
\begin{equation}\label{4}
    f(\Omega(p,p_1,\theta))=f_1(\Omega(p,p_1,\theta))=\dfrac{1}{c}g(\mathbb{D})\subset f(\mathbb{D}).
\end{equation}
Since $g$ is analytic in the unit disk and has a simple pole at $z=-1\in\partial\mathbb{D}$, by Corollary \ref{s1-thm-012}, the radius of the largest schlicht disk contained in $g(\mathbb{D})$ is not finite. Consequently, $B(\lambda,\mu)$ is not finite. \\

\textbf{Case 2: The poles $\lambda$ and $\mu$ lie on the same radial line.} In this case, without the loss of generality, we assume $\lambda=p$ and $\mu=q$ with $0 < p < q < 1$.
For $f\in \mathcal{M}_1(p,q)$, let
$$
F_f(z)=\dfrac{(f\circ\psi)(z)}{(f\circ\psi)'(0)},\quad\text{where}~ \psi(z)=\dfrac{p+q-2z}{2-(p+q)z}
$$
is an automorphism of the unit disk $\mathbb{D}$. If
$$
z_1:=\dfrac{p-q}{2-q(p+q)}\quad\text{and}\quad z_2:=\dfrac{q-p}{2-p(p+q)}
$$
then a simple computation shows that $\psi(z_1)=q$ and $\psi(z_2)=p$. Hence, the function $F_f$ is analytic in $\mathbb{D}$ and has simple poles at the points $z=z_1,z_2$ with
$$
F_f(\mathbb{D})=\frac{1}{(f\circ\psi)'(0)} f(\mathbb{D}).
$$
Since $z_1<0$ and $z_2>0$, the poles of $F_f$ at $z_1$ and $z_2$ are situated on opposite sides of the origin, effectively placing them on different radial lines. By applying the result of \textbf{Case 1}, the image $F_f(\mathbb{D})$ contains disks of arbitrary radii. Because $f(\mathbb{D})$ is a scaled version of $F_f(\mathbb{D})$, it follows that $B(p,q)$ is likewise infinite.

\end{proof}

\begin{cor}\label{s1-thm-016}
For any distinct $\lambda,\mu\in\overline{\mathbb{D}}\setminus\{0\}$, the Landau constant $L(\lambda,\mu)$ for the class $\mathcal{M}_1{(\lambda,\mu)}$ is infinite.
\end{cor}

\noindent\textbf{Declarations:\\}

\noindent\textbf{Data availability:}
Data sharing not applicable to this article as no data sets were generated or analyzed during the current study.\\

\noindent\textbf{Authors Contributions:}
All authors contributed equally to the investigation of this problem. The order of authorship is listed alphabetically by surname. All authors have read and approved the final manuscript.\\


\noindent\textbf{Conflict of interest:} The authors declare that they have no conflict of interest.

\end{document}